\documentclass[12pt]{article}

\date{}          

\usepackage{amssymb}
\usepackage{verbatim}
\usepackage{amssymb}
\usepackage{amsmath}
\usepackage{mathdots}
\usepackage{amsthm}
\usepackage{epsfig}
\usepackage[all]{xy}

\usepackage{rawfonts}
\usepackage{bm}
\usepackage{graphicx}
\usepackage{multirow}
\usepackage[usenames,dvipsnames]{pstricks}
\usepackage{pstricks}
 \usepackage{epsfig}
 \usepackage{pst-grad} 
 \usepackage{pst-plot}
\usepackage[T1]{fontenc}

\newcommand{\be}{\begin{equation}}
\newcommand{\ee}{\end{equation}}
\newcommand{\bi}{\begin{itemize}}
\newcommand{\ei}{\end{itemize}}
\newcommand{\bea}{\begin{eqnarray*}}
\newcommand{\eea}{\end{eqnarray*}}
\newcommand{\ba}{\begin{array}}
\newcommand{\ea}{\end{array}}

\newcommand{\ZZ}{\mathbb{Z}}
\newcommand{\RR}{\mathbb{R}}
\newcommand{\NN}{\mathbb{N}}

\newcommand{\linf}{\ell_{\infty}(\ZZ)}

\newcommand{\bs}{\bar{s}}

\newcommand{\supp}{\text{supp}}
\newtheorem{theorem}{Theorem}[section]
\newtheorem{lemma}[theorem]{Lemma}
\newtheorem{proposition}[theorem]{Proposition}

\newtheorem{remark}[theorem]{Remark}

\numberwithin{equation}{section}

\title{B-spline normal multi-scale transforms for planar curves}
\author{Stanislav Harizanov \thanks{The author is partially supported by the Bulgarian fund "Nauchni Izsledvania" under grant DFNI-T01/0001. Project title: "Effective methods and algorithms for geometric modeling".} \\ TU Kaiserslautern}

\begin{document}
\maketitle

\begin{abstract}
Normal multi-scale transform \cite{guskov00normal} is a nonlinear multi-scale transform for representing geometric objects that has been recently investigated \cite{DRS,HOS,Oswald}. The restrictive role of the exact order of polynomial reproduction $P_e$ of the approximating subdivision operator $S$ in the analysis of the $S$ normal multi-scale transform, established in \cite[Theorem 2.6]{HOS}, significantly disfavors the practical use of these transforms whenever $P_e\ll P$. We analyze in detail the normal multi-scale transform for planar curves based on B-spline subdivision scheme $S_p$ of degree $p\ge3$ and derive higher smoothness of the normal re-parameterization than in \cite{HOS}. We show that further improvements of the smoothness factor are possible, provided the approximate normals are cleverly chosen. Following \cite{Oswald}, we introduce a more general framework for those transforms where more than one subdivision operator can be used in the prediction step, which leads to higher detail decay rates.
\\[1ex]
\\[1ex]
\end{abstract}

\section{Introduction}\label{sec1}
Let $\mathcal{C}$ be a continuous planar curve, $\mathbf{v}^0$ be a finite sequence of distinct points on it, and $S$ be a univariate, local, shift-invariant subdivision operator. Starting with $\mathbf{v}^0$ we recursively generate finer scale representations $\mathbf{v}^j$, $j\ge1$ of $\mathcal{C}$, applying first the operator $S$ component-wise to the coarse-scale sequence $\mathbf{v}^{j-1}$ and then adjusting the fine-scale data so that they lie on the curve. The latter is done via moving the predicted points $S\mathbf{v}^{j-1}$ along a priori given directions $\hat{\mathbf{n}}^j$ (again generated from the coarse-scale data $\mathbf{v}^{j-1}$) until $\mathcal{C}$ is hit. Analytically, one-level refinement is described by
\be\label{eq:Normal MT}
\mathbf{v}^{j}=S\mathbf{v}^{j-1}+d^{j}\hat{\mathbf{n}}^j,\qquad\forall j\in\NN,
\ee
and is illustrated on Fig.~\ref{fig:Normal}. Such a transform decomposes $\mathcal{C}$ into its multi-scale components $\{\mathbf{v}^0,d^1,d^2,\dots\}$, and vice versa allows for complete reconstruction of $\mathcal{C}$ from them. It is called {\em normal}, because whenever $\mathcal{C}\in G^1$, the unit vector $\hat{\mathbf{n}}^j_i$ mimics the true normal of $\mathcal{C}$ at $\mathbf{v}^j_i$ for all admissible $i$. Unlike classical multi-scale transforms (MTs), the normal MTs for curves deal with scalar detail sequences $d^j$, allowing more efficient computational processing of the data. On the other hand, the different nature of data and details together with the nonlinear dependency of $\hat{\mathbf{n}}^j$ on $\mathbf{v}^{j-1}$ make the $S$ normal MT nonlinear, even when $S$ is linear, and thus its mathematical analysis is nontrivial and nonclassical.

\begin{figure}[htp]
\begin{center}
\includegraphics[width=0.6\textwidth]{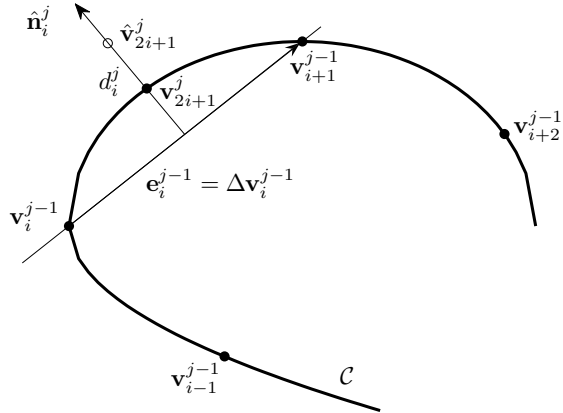}
\caption{One step of a normal MT for interpolatory $S$} \label{fig:Normal}
\end{center}
\end{figure}

In \cite{HOS}, such an analysis for linear $S$ has already been performed and the general result there (Theorem 2.6) states that both the detail decay rate and the smoothness of the normal re-parametrization of the normal MT depend on the exact order of polynomial reproduction $P_e$ of $S$.
While the appearance of $P_e$ in the detail analysis is natural, this is not the case for the smoothness analysis, where intuitively only the smoothness of the underlined objects $\mathcal{C}$ and $S$ should matter.
The main reason for the appearance of $P_e$ there is that the proof of \cite[Theorem 2.6]{HOS} is based on a bootstrapping argument, where detail and smoothness analysis are tightly coupled via Lemmas 3.4 and 3.5, even though it is well known \cite{KSS:geometryCompression} that geometry and parameter information are locally disconnected. So is $P_e$ indeed a true constraint for the smoothness properties of the normal re-parametrization?

In this paper we analyze the normal MTs associated to the family of subdivision operators $\{S_p\}_{p\ge3}$, generating B-splines of degree $p$ in the limit. Our research is motivated by several reasons: Firstly,
$P_e=2$ independently on $p$ and such a case study may help for investing the exact role of $P_e$ in the general smoothness analysis. Secondly, the restrictive role of $P_e=2$ significantly disfavors the practical use of these transforms, while at the same time because of their nice geometric and analytic properties, the $S_p$ operators are still among the most favorable schemes considered in applications. Finally, our work can be viewed as a continuation of \cite{Runborg, Harizanov}, where the $S_1$ and the $S_2$ normal MTs have been analyzed, respectively.

We prove $C^{2,1-}$ guaranteed smoothness of the normal re-parameterization of the $S_p$ normal MT, $p\ge3$, improving on the $C^{1,1}$ smoothness derived from the general theory in \cite{HOS}. Furthermore, using special normal directions $\hat{\mathbf{n}}^j$, namely those generated from $\Delta\mathbf{v}^{j-1}$ via the operator $S_{p-2}$, we manage to further increase that smoothness to $C^{3,1-}$ for the what we will call $(S_p,S_{p-2})$ normal MTs, $p\ge4$. Last, but not least, following \cite{Oswald}, we introduce the notion of combined transforms and show that such $(S_p,S_{p-2},T)$ normal MTs, $p\ge3$, lead to detail decay rate 4, provided $T$ is suitably chosen. For comparison, the optimal detail decay rate for the pure $S_p$ normal MT is 2.

The paper is organized as follows: In Section~\ref{sec2} the main definitions and analytical tools that will be later used are explained. In Section~\ref{sec3} the smoothness analysis for the $S_p$ normal MTs is performed. In Section 4, combined normal MTs are defined and analysis on the detail decay rate of the $(S_p,S_{p-2},T)$ normal MTs is executed.


\section{Notation, framework, and preliminaries}\label{sec2}

Throughout this paper boldfaced letters are used for vector-valued quantities
(points in $\RR^2$, sequences of such points, or $\RR^2$-valued functions). The notation $|\cdot|$ is used both for the absolute value in $\RR$ and the Euclidean norm in $\RR^2$ which should be clear from the context, while $\|\cdot\|$ is used for both the $\ell_{\infty}$-norm and its associated operator norm. We use also $\|\cdot\|_I$ for the norm of the corresponding subsequence, restricted to the index set $I\subset\ZZ$. The scalar product of two vectors $\mathbf{a}, \mathbf{b}$ in $\RR^2$ is simply denoted by $\mathbf{a}\mathbf{b}$. If $a,b$ are sequences in $\RR$, and $\mathbf{b}$ is a sequence in $\RR^2$ then $ab:=(a_i b_i)_{i\in\ZZ}$, $a\mathbf{b}:=(a_i\mathbf{b}_i)_{i\in\ZZ}$ for short.

As in \cite{HOS}, let $\mathcal{C}$ be a rectifiable, non-self-intersecting, closed planar curve, given via its regular $C^1$ arc-length parametrization $\mathbf{v}:[0,L]\to\RR^2$,
with $L$ being the length of $\mathcal{C}$. Closedness of $\mathcal{C}$ avoids dealing with boundaries, while self-intersecting curves can be partitioned into (finitely many) non-self-intersecting components, meaning that both these assumptions are not crucial for the future analysis but just conveniently simplify the setup. Under ``smoothness of $\mathcal{C}$'' we understand the one of $\mathbf{v}(s)$, and we use the H{\"o}lder spaces $C^{k,\alpha}$, with $k\in\NN,\; \alpha\in(0,1]$. The role of the smoothness of $\mathcal{C}$ in the analysis of the corresponding normal MTs has already been established \cite{DRS,HOS} and we do not elaborate in this direction. Although in the following sections different spaces $C^{k,\alpha}$ will appear, we introduce them only for the sake of completeness and for stating the results in their full generality. However, for the purposes of this paper the reader may always think of $\mathcal{C}$ as being ``smooth enough'' or even $C^{\infty}$.

Given a finite number $N$ of distinct points on $\mathcal{C}$: $\mathbf{v}^0_i=\mathbf{v}(s^0_i)$, $i=0,\ldots,N-1$, such that
$s_0^0<s_1^0<\ldots<s_{N-1}^0< s_0^0+L$, we extend both the sequences $\mathbf{v}^0$ and $s^0$ periodically over $\ZZ$ via $\mathbf{v}^0_{i+N}=\mathbf{v}_i^0$ and  $s^0_{i+N}=s^0_i+L$ for all $i\in \ZZ$. The subdivision operator $S_p:\linf\to\linf$ is defined in a recursive way, using the Lane-Riesenfeld \cite{Lane-Riesenfeld} algorithm:
\be\label{eq:Sp}
(S_0x)_{2i}=(S_0x)_{2i+1}=x_i,\qquad (S_p x)_i=\frac{(S_{p-1} x)_i+(S_{p-1}x)_{i+1}}{2},\quad i\in\ZZ.
\ee
For the point-mass sequence $\delta$ we have that the limit function $S_p^{\infty}\delta$ is the cardinal B-spline of degree $p$ with support $[0,p+1]$ and thus $S_p$ generates
$C^{p-1,1}$ limits.
The order of polynomial reproduction $P$ of $S_p$ is $P=p+1$, meaning that for any polynomial $\mathfrak{p}\in\Pi_{p}$ of degree at most $p$ there exists a polynomial $\mathfrak{q}\in\Pi_{p-1}$ with $\deg(\mathfrak{q}) < \deg(\mathfrak{p})$ such that
\be\label{PDef}
S_p(\mathfrak{p}|_{\ZZ}) = \mathfrak{p}|_{2^{-1}\ZZ} + \mathfrak{q}|_{2^{-1}\ZZ}.
\ee
Here, $\mathfrak{p}|_{\ZZ}$ denotes the sequence $(\mathfrak{p}(i))_{i\in\ZZ}$, and $\mathfrak{p}|_{2^{-1}\ZZ}$ the sequence
$(\mathfrak{p}(i/2))_{i\in\ZZ}$. On the other hand, (\ref{eq:Sp}) guarantees a positive mask for $S_p$, which leads to an exact order of polynomial reproduction $P_e=2$ (see \cite{HOS}), meaning that for any linear polynomial $\mathfrak{p}$
\be\label{PeDef}
S(\mathfrak{p}|_{\ZZ}) = \mathfrak{p}|_{2^{-1}\ZZ+c_{S_p}}=(\mathfrak{p}(i/2+c_{S_p}))_{i\in\ZZ},
\ee
with some shift parameter $c_{S_p}$. The following properties of $S_p$ are instrumental for our analysis:

\begin{lemma}\label{lemma:Sp}
\begin{itemize}
\item[(i)] The derived subdivision operators of $S_p$ are characterized via
\be\label{eq:DerivedSp}
S^{[q]}_p=\frac1{2^{q}}S_{p-q},\qquad\forall p\ge q\ge0.
\ee
\item[(ii)] For the shift parameter $c_{S_p}$ we have
\be\label{eq:c(Sp)}
c_{S_p}=\frac{p-1}{4},\qquad\forall p\ge1,
\ee
while for the quadratic and cubic polynomial samples, and $p\ge2$ and $p\ge3$ we have
\be\label{eq:P=2 and 3}
S_p(t^2|_{\ZZ})=\left(t^2+\frac{p+1}{16}\right)\Big|_{2^{-1}\ZZ+c_{S_{p}}};\qquad
S_p(t^3|_{\ZZ})=\left(t^3+\frac{3(p+1)}{16}t\right)\Big|_{2^{-1}\ZZ+c_{S_{p}}},
\ee
respectively.
\end{itemize}
\end{lemma}
\begin{proof}
For $(i)$, recall that the first derived scheme $S_p^{[1]}$ is defined via the identity $S_p\circ\Delta=\Delta\circ S_p^{[1]}$, where $\Delta x_i=x_{i+1}-x_i$ is the forward finite difference operator, and recursively $S_p^{[q]}$ satisfies $S_p\circ\Delta^q=\Delta^q\circ S_p^{[q]}$. We prove (\ref{eq:DerivedSp}) for $q=1$, using induction on $p$. Since
$$(S^{[1]}_1\Delta x)_{2i}=(S^{[1]}\Delta x)_{2i+1}=\frac{\Delta x_i}{2}\quad\Longrightarrow\quad S^{[1]}_1=\frac12 S_0,$$
and assuming that $S^{[1]}_p=S_{p-1}/2$ we deduce from (\ref{eq:Sp}) that
$$(S^{[1]}_{p+1}\Delta x)_i=(\Delta S_{p+1}x)_i=\frac{(\Delta S_p x)_i+(\Delta S_p x)_{i+1}}{2}=\frac{\frac12S_{p-1}\Delta x_i+\frac12S_{p-1}\Delta x_{i+1}}{2}=\left(\frac12S_p\Delta x\right)_i.$$
Applying another induction argument, this time on $q$, we complete the proof.

For the proof of $(ii)$ we again heavily use (\ref{eq:Sp}). For the first part, we consider the linear polynomial sample $t|_{\ZZ}$, and since $P_e=2$ and $c_{S_1}=0$ ($S_1$ is interpolatory) (\ref{eq:c(Sp)}) follows from
\bea
t|_{1/2\ZZ+c_{S_{p+1}}}=S_{p+1}(t|_{\ZZ})=\frac{t|_{1/2\ZZ+c_{S_{p}}}+t|_{1/2\ZZ+c_{S_{p}}+1/2}}{2}=t|_{1/2\ZZ+c_{S_p}+1/4}\quad\Rightarrow\quad c_{S_{p+1}}=c_{S_p}+\frac14.
\eea
To show (\ref{eq:P=2 and 3}) we also use the central symmetry in the mask of $S_p$ and the symmetry of both $t^2$ and $t^3$ with respect to the origin. Following the notation in (\ref{PDef}) we conclude that for $\mathfrak{p}(t)=t^2$, $\mathfrak{q}(t)=c_1$, while for $\mathfrak{p}(t)=t^3$, $\mathfrak{q}(t)=c_2t$, where both $c_1$ and $c_2$ are functions only on $p$. Deriving that $c_1=(p+1)/16$ and $c_2=3(p+1)/16$ is an easy exercise that is left to the reader.
\end{proof}

Some remarks are in order. Generally, it is enough to deal with only two shift parameters, namely $c_S=0$ for primal schemes, and $c_S=1/4$ for dual ones. The result in Lemma~\ref{lemma:Sp} does not contradict with that and is affected by the recursive formulation of $S_p$, which shifts the refined data to the right, e.g., $(S_3x)_{2i}=(x_i+x_{i+1})/2=(S_1x)_{2i+1}$. In standard definitions the centered versions $S'_p$ of the operators, e.g., $(S'_3x)_i=(S_3x)_{i-1}$ are usually used. However, (\ref{eq:Sp}) is in the core of our analysis and we prefer to work with such non-centered formulation of $S_p$, instead. Due to linearity, direct computations show that (\ref{eq:P=2 and 3}) remains also valid for arbitrary shifts $\alpha\in\RR$ of the data:
\be\label{eq:P=2 and 3 generalized}
\begin{split}
S_p((t+\alpha)^2|_{\ZZ})&=S_p(t^2|_{\ZZ+\alpha})=\left(t^2+\frac{p+1}{16}\right)\Big|_{2^{-1}\ZZ+\alpha+c_{S_{p}}};\\
S_p((t+\alpha)^3|_{\ZZ})&=S_p(t^3|_{\ZZ+\alpha})=\left(t^3+\frac{3(p+1)}{16}t\right)\Big|_{2^{-1}\ZZ+\alpha+c_{S_{p}}}.
\end{split}
\ee
Note that Lemma~\ref{lemma:Sp}(ii) can be further extended to
$$S_p(t^n|_{\ZZ})=(t^n+A_{2}t^{n-2}+A_{4}t^{n-4}+\dots)|_{2^{-1}\ZZ+c_{S_{p}}},\qquad\forall n\le p,$$
and the coefficients $A_{2i}$, $i=1,\dots,[n/2]$, can be explicitly computed. For example,
\bea
A_2=\frac{{n\choose 2}}{4^2}(p+1);\quad A_4=\frac{{n\choose 4}}{4^4}(3p+1)(p+1);\quad
A_6=\frac{{n\choose 6}}{4^6}(15p^2+1)(p+1).
\eea

We denote by $I:=[0,p+1]$ the invariant neighborhood of $S_p$, meaning that all $(Sx)_i$ with indices in $[0,1]+I=[0,p+2]$ are computable solely from values $x_i$ with indices $i\in I$. Even though $I$ depends on $p$ we prefer not to indicate it explicitly, since fixing $S_p$ determines $I$ uniquely. We also use the notation $I_i=i+I$ for the invariant neighborhood of index $i$.

For the choice of the normal directions we assume that, with some fixed constant $C_1$, the sequence of approximate normals $\hat{\mathbf{n}}$
associated with $\hat{\mathbf{v}}=S\mathbf{v}$ satisfies
\be\label{Normal}
 \hat{\mathbf{n}}_{K}={\mathbf{n}}(\xi_{K}):\qquad  |\xi_{K}-s_{k}|\le C_1\|\Delta s\|_{I_i},\quad k\in I_i, \quad K\in I_{2i}\cup I_{2i+1}.
\ee
with some $\xi_K$ for all $K\in I_{2i}\cup I_{2i+1}$ and all $i\in \ZZ$. Note that this should hold for all the scales $j$ in the normal MT, even though it is not explicitly mentioned in the formula.

In this paper, like in \cite{HOS}, we are only interested in the decomposition step of the $S_p$ normal MT, and mainly in its smoothness analysis. By smoothness of the normal MT we understand the one of the induced regular parameterization $\mathbf{v}(t)=\lim_{j\to\infty}\mathbf{v}^j(t)$ of $\mathcal{C}$ (called {\em normal re-parametrization}), provided the limit exists and is continuous, where $\mathbf{v}^j(t)$ is the linear interpolant for the data vector $\mathbf{v}^j$ at the grid $2^{-j}\ZZ$ (i.e., $\mathbf{v}^j(2^{-j}i)=\mathbf{v}^j_i,\;\forall i$). Since no other parametrization of $\mathcal{C}$ can be smoother than $\mathbf{v}(s)$, it is equivalent to analyze the regularity properties of $s(t)=\lim_{j\to\infty}s^j(t)$, where $s^j(t)$ is the corresponding arc-length linear interpolant. Therefore, we consider
 \be\label{eq:Normal MT for s}
 s^j=S_ps^{j-1}+\omega^j,\qquad \omega^j:= s^j-S_ps^{j-1}.
 \ee
Investigating (\ref{eq:Normal MT for s}) rather than (\ref{eq:Normal MT}) is beneficial because not only it gives rise to a standard/classical MT for $s$ but also it allows us to locally decouple geometry and parameter information. To show that, we first write the Taylor formula for $\mathbf{v}(s)\in C^{k,\alpha}$ in the form
\be\label{Taylork}
\mathbf{v}(s)=\sum_{n=0}^k \frac{\mathbf{v}^{[n]}(\xi)}{n!}(s-\xi)^n +\mathbf{r}(s,\xi),\qquad
|\mathbf{r}(s,\xi)|\le C_{\mathcal{C}}|s-\xi|^{k+\alpha}.
\ee
If $k\ge2$ we define the local frame at $\xi$ as the one, centered at $\mathbf{v}(\xi)$ and given by unit tangent $\mathcal{T}=\mathbf{t}(\xi)$ and normal vectors $\mathcal{N}=\mathbf{n}(\xi)$
to the curve. We call $x(s),y(s)$ the local frame coordinates
\be\label{LF}
\mathbf{v}(s)=\mathbf{v}(\xi)+x(s)\mathcal{T}+y(s)\mathcal{N}
\ee
of $\mathcal{C}$. To keep the notation as simple as possible, we chose not to indicate explicitly the dependence of $x(s),y(s)$ on $\xi$. However, this has to be kept in mind. Then, the Frenet-Serret formulae read $\mathcal{T}'=\mathrm{k}(\xi)\mathcal{N},\;\mathcal{N}'=-\mathrm{k}(\xi) \mathcal{T}$, where the scalar function $\mathrm{k}(s)$ corresponds to the signed length of the curvature vector $\mathbf{v}''(s)$ of $\mathcal{C}$ at $s$. Combining them with (\ref{Taylork}) and letting $k=3$, we derive
\be\label{Taylor expansion for s}
\begin{split}
x(s)&=(s-\xi)-\frac{\mathrm{k}^2(\xi)}{6}(s-\xi)^3+\mathbf{r}(s,\xi)\mathcal{T}=(s-\xi)-\frac{\mathrm{k}^2(\xi)}{6}(s-\xi)^3+\mathbf{r}_1(s,\xi),\\
y(s)&=\frac{\mathrm{k}(\xi)}{2}(s-\xi)^2+\frac{\mathrm{k}'(\xi)}{6}(s-\xi)^3+\mathbf{r}(s,\xi)\mathcal{N}=
\frac{\mathrm{k}(\xi)}{2}(s-\xi)^2+\frac{\mathrm{k}'(\xi)}{6}(s-\xi)^3+\mathbf{r}_2(s,\xi).
\end{split}
\ee
Since $\mathbf{r}^2_1(s,\xi)+\mathbf{r}^2_2(s,\xi)=\mathbf{r}^2(s,\xi)$, $|\mathbf{r}_{1,2}(s,\xi)|\le  C_{\mathcal{C}}|s-\xi|^{3+\alpha}$. The decoupling, mentioned above, takes place when we choose $\xi=\xi^j_K$ as in (\ref{Normal}). The approximate normal $\hat{\mathbf{n}}^j_K$ has no tangential component. The size of the detail $d^j_K$ is irrelevant for the computation of $x(s)$, which is a third order perturbation of a linear function, thus locally invertible around $\xi^j_K$. To summarize, we can estimate $s^j_K$ without explicitly using $d^j_K$, but only using the direction information $\hat{\mathbf{n}}^j_K$.

The convergence and smoothness analysis of the MT (\ref{eq:Normal MT for s}) is well-known and is based on the following perturbation result (see \cite[Theorem 3.3]{DRS} for the proof in the general setting).

\begin{lemma}\label{lemma:DRS}
If there are constants $C$ and $\gamma>0$ such that the offsets $\omega^j$ in (\ref{eq:Normal MT for s}) satisfy
\be\label{eq:DRS}
\|\omega^j\|\le C2^{-\gamma j}, \qquad j\ge1,
\ee
then $s^j(t)$ converges uniformly to $s(t)$, and $s(t)\in C^{k',\alpha'}$ whenever $k'+\alpha'<\min(p,\gamma)$.
\end{lemma}

The classical way to obtain estimations of the type (\ref{eq:DRS}) is via proximity. The following lemma, which we state without proof, is central for the analysis in \cite{HOS} and we will also need it for this paper.
\begin{lemma}(\cite[Lemma 3.2]{HOS})\label{lem41}
For given $S$ with $P_e\ge 2$, let $1\le M < P_e$, and assume that $F: [a,b]\to \RR$ is $C^{M,\rho}$ for some $0<\rho\le 1$. Then for any finite sequence $(s_l)_{l\in I}\subset [a,b]$ and any index $K$ such that $Ss_K\in [a,b]$ is well-defined we have
$$
|F(Ss_K) - SF(s)_K| \le C (\sum_{\nu\in E_M} \prod_{m=1}^M\|\Delta^m s\|_{I}^{\nu_m} + \|\Delta s\|_{I}^{M+\rho}),\qquad K\in I\cup I_{1},
$$
where the constant $C$ is independent of the sequence $s$ and $K$, and
$$
E_M:=\{\nu\in \ZZ^M_+:\quad \sum_{m=1}^M m\nu_m = M+1,\quad 2\le \sum_{m=1}^M \nu_m\le M\}.
$$
\end{lemma}

The next result is again well-known and easy to show.
\begin{lemma}\label{lemma:Oswald}
Let $S,T:\linf\to\linf$ be local, bounded, linear operators that exactly reproduce polynomials of order $P_e$ with $c_S=c_T$. Then there is a constant $C$, such that
\be\label{eq:Oswald}
\|Sx-Tx\|\le C\|\Delta^{P_e}x\|,\qquad\forall x\in\linf.
\ee
\end{lemma}

\section{Smoothness analysis}\label{sec3}
\begin{theorem}\label{thm:Smoothness1}
Let $p\ge3$. Let $\mathcal{C}$ be $C^3$, $\mathbf{v}^0\in\mathcal{C}$. Let $\hat{\mathbf{n}}^j$, $j\ge1$, satisfy (\ref{Normal}) with $s$ replaced by $s^{j-1}$, and the additional $|s^j_K-\xi_K|\le C_1\|\Delta s^{j-1}\|_{I_i}$, with a constant $C_1$ independent of $i$ and $j\ge1$. If the $S_p$ normal MT for $\mathcal{C}$ with initial data $\mathbf{v}^0$ and normals $\mathbf{n}^j$ is well-posed and convergent, then the normal re-parametrization is $C^{2,1-}$.
\end{theorem}
\begin{proof}
Firstly, we consider only one refinement step $\bar{\mathbf{v}}=S_p\mathbf{v}+d\hat{\mathbf{n}}$, respectively $\bar{s}=S_ps+\omega$, and we skip the index $j$ for simplicity. Following the notation from (\ref{Normal}), take any $K\in I_{2i}\cup I_{2i+1}$ and consider the local frame at $\xi_K$. Applying (\ref{Taylor expansion for s}) we get for all $s\in[0,L]$
$$x(s)=(s-\xi_K)+\mathbf{r}_1(s,\xi_K),\qquad |\mathbf{r}_1(s,\xi_K)|\le C_{\mathcal{C}}|s-\xi_K|^{3}.$$
Since $\hat{\mathbf{n}}_K$ has no tangential component $x(\bar{s}_K)=(S_px)_K$. Subtracting $x((S_ps)_K)$ from both sides, using that $S_p$ exactly reproduces linear functions and has positive mask, thus $\|S_p\|=1$, we deduce
\bea
&&x(\bs_K)-x((S_ps)_K)=(S_px)_K-x((S_ps)_K).\\
x(\bs_K)-x((S_ps)_K)&=&\omega_K+\mathbf{r}_1(\bs_K,\xi_K)-\mathbf{r}_1((S_ps)_K,\xi_K);\\
(S_px)_K-x((S_ps)_K)&=&\left(S_p\Big((s-\xi_K)+\mathbf{r}_1(s,\xi_K)\Big)\right)_K-\Big((S_ps)_K-\xi_K)+\mathbf{r}_1((S_ps)_K,\xi_K)\Big)\\
&=&(S_ps)_K-\xi_K+(S_p\mathbf{r}_1(s,\xi_K))_K-(S_ps)_K+\xi_K-\mathbf{r}_1((S_ps)_K,\xi_K);\\
\Longrightarrow\quad |\omega_K|&=&|(S_p\mathbf{r}_1(s,\xi_K))_K-\mathbf{r}_1(\bs_K,\xi_K)|\le C_{\mathcal{C}}\|s-\xi_K\|^3_{I_i}+C_{\mathcal{C}}|\bs_K-\xi_K|^3\\
&\le& 2C_{\mathcal{C}}C^3_1\|\Delta s\|^3_{I_i}\le C\|\Delta s\|^3_{I_i}.
\eea
The above computations were independent of the level $j$, so there is a constant $C=2C_{\mathcal{C}}C^3_1$ such that for all $j\ge1$
\be\label{eq:omega3}
\|\omega^j\|\le C\|\Delta s^{j-1}\|^3.
\ee
Now, in order to apply Lemma~\ref{lemma:DRS} and complete the proof we just need to show $\|\Delta s^j\|\asymp 2^{-j}$. Such estimations appear quite often in the literature \cite{DRS,Runborg,HOS,Harizanov}, so we only sketch the argument here. Since the $S_p$ normal MT is convergent, $s^j(t)$ uniformly converges towards a continuous $s(t)$ as $j\to\infty$. Therefore, $\|\Delta s^j\|\to 0$ and there exists a $j_0$ with $C\|\Delta s^j\|^2<1/2$ for all $j\ge j_0$. Without loss of generality, let $j_0=0$. From (\ref{eq:Normal MT for s}), (\ref{eq:DerivedSp}), (\ref{eq:omega3}), and $\|S_p\|=1$
\bea
\|\Delta s^{j}\|\le\frac12\|\Delta s^{j-1}\|+C\|\Delta s^{j-1}\|^3&\le&\underbrace{(1/2+C\|\Delta s^{j-1}\|^2)}_{\rho}\|\Delta s^{j-1}\|\quad\Rightarrow\quad
\|\Delta s^{j}\|\le \rho^j\|\Delta s^0\|\\
\Longrightarrow\quad\|\Delta s^{j}\|&\le&\prod_{l=0}^{j-1}(1/2+C\rho^{2l})\le2^{-j}\prod_{l=0}^{\infty}(1+2C\rho^{2l})\le C2^{-j}.
\eea
The constant $C$ changes from line to line, but remains uniformly bounded with respect to $j$! The notation ``$C^{2,1-}$'' is consistent with the result in Lemma~\ref{lemma:DRS} and means that $s(t)\in C^{2,\alpha}$ for all $\alpha<1$.
\end{proof}

\begin{remark}\label{remark1}
Under the assumptions in Theorem~\ref{thm:Smoothness1} for the $S_p$ normal MT, $p\ge3$, there is a constant $C_2$ such that for every $j\ge0$, every $i\in\ZZ$, and every $k\in I_i$
\be\label{GridSmth2}
s^j_k=s^j_i+(k-i)\|\Delta s^j\|_{I_i}+\mathbf{r}_{k,i}, \qquad |\mathbf{r}_{k,i}|\le C_22^{-2j}.
\ee
\end{remark}
Indeed, using the same technique as in the proof above we derive $\|\Delta^2 s^{j}\|\le C''2^{-2j}$ for all $j\ge0$. From the reverse triangle inequality $\|\Delta s^j\|\ge(1/2-C\|\Delta s^{j-1}\|^2)\|\Delta s^{j-1}\|$, we conclude that there exists a $j_0\ge0$, such that $\|\Delta s^j\|\ge C'2^{-j}$ for all $j\ge j_0$. Finally, the $S_p$ normal MT is well-posed, thus $\Delta s^j >\mathbf{0}$, $j=0,\dots,j_0-1$, meaning that $\|\Delta s^j\|\ge C'2^{-j}$ with a possibly smaller, but still bounded away from zero, constant $C'$. Therefore, $\|\Delta^2 s^j\|\le C''/(C')^2\|\Delta s^j\|^2$ and (\ref{GridSmth2}) follows from \cite[Lemma 2.5]{HOS}.

Theorem~\ref{thm:Smoothness1} improves the smoothness result from \cite[Theorem 2.6]{HOS} for the family $\{S_p\}_{p\ge3}$ of prediction operators. There is experimental evidence (see \cite{HOS}) that, when generating $\hat{\mathbf{n}}^j$ via random (admissible) $\xi^j$, $s(t)\in C^{2,1-}$ is the best one can hope for. However, a more systematic and ``clever'' choice of approximate normals may lead to higher regularity. The remainder of Section~\ref{sec3} is devoted to this idea.

First of all, $S_p\mathbf{v}^j_K$ is a convex combination of elements of $\mathbf{v}^{j-1}$, thus the predicted point and the initial curve are separated by the graph of $\mathbf{v}^{j-1}(t)$, whenever the curvature $\mathrm{k}(s)$ of $\mathcal{C}$ has a constant sign on the local arc $\mathbf{v}(s^{j-1}|_{I_i})$. Since $\mathbf{v}^{j-1}(t)$ is piece-wise linear, it is a second order approximation of $\mathcal{C}$, meaning that $d^j_K\asymp 2^{-2j}$. For $C^3$ curves $\mathcal{C}$, $\mathrm{k}$ is continuous. Hence, for every non-flat point $\mathbf{v}(s)$ with $\mathrm{k}(s)\neq0$ there exists a neighborhood $[s^-,s^+]$ of $s$, where the curvature does not vanish and keeps its sign. In other words, asymptotically all the nontrivial details $d^j_K$ are of order 2 regardless of the choice of $\hat{\mathbf{n}}^j_K$. Big details lead to big displacements of the newly inserted points, thus the choice of $\xi^j$ in (\ref{Normal}) has indeed an impact on the regularity of the generated data.

Formula (\ref{Normal}) emphasizes the role of $\mathcal{C}$ and is to some extend misleading because the main property of $\hat{\mathbf{n}}^j$ in a general $S$ normal MT is that they depend solely on the coarse-scale data (which also encodes curve information, but in a discrete way). We define a special class of generalized normals with the help of another linear operator $N:\linf\to\linf$ via
\be\label{Normal2}
\hat{\mathbf{n}}_K^{j}=\frac{(N\Delta \mathbf{v}^{j-1})_K^{\perp}}{|(N\Delta \mathbf{v}^{j-1})_K|},\qquad K\in I_{2i}\cup I_{2i+1},
\ee
where $N$ reproduces constants, has positive mask and $\supp(N)\subseteq \supp(S^{[1]})$. Obviously (\ref{Normal2}) implies (\ref{Normal}). Indeed, $(N\Delta \mathbf{v}^{j-1})_K$ is a convex combination of $\{\Delta\mathbf{v}^{j-1}_k:k\in I_i\}$. Hence, there exists a $\xi^j_k$ on the interval $s^{j-1}|_{I_i}$ such that $\mathbf{v}'(\xi^j_k)\parallel (N\Delta \mathbf{v}^{j-1})_K$, making $\hat{\mathbf{n}}_{K}^j=\mathbf{n}(\xi^j_K)$ and $\|\xi^j_{K}-s^{j-1}\|_{I_i}\le |I_i|\|\Delta s^{j-1}\|_{I_i}$. Furthermore, the normal line $\ell^j_K(t):=S\mathbf{v}^{j-1}_K+t\hat{\mathbf{n}}^j_K$ splits the set $\mathbf{v}^{j-1}_{I_i}$, leaving elements in both the hyperplanes. Introducing a simple criteria that in case of many intersection points between $\ell^j_K$ and $\mathcal{C}$ we always take a $\mathbf{v}^j_K$ that corresponds to arc-length parameter $s^j_K$ on the interval $s^{j-1}|_{I_i}$ (such a point exists, due to continuity of $\mathcal{C}$), we guarantee
\be\label{eq:local imputation}
|\xi^j_{K}-s^j_K|\le |I_i|\|\Delta s^{j-1}\|_{I_i}.
\ee
Due to (\ref{Normal2}) we can talk about {\em $(S,N)$ normal MT}, meaning that the prediction points are computed by $S$, while the normal directions are generated by $N$. Due to (\ref{eq:local imputation}), for any admissible choice of $N$ the assumption $|s^j_K-\xi_K|\le C_1\|\Delta s^{j-1}\|_{I_i}$ in Theorem~\ref{thm:Smoothness1} is automatically fulfilled with $C_1=p+1$.

Finding the optimal choice of $N$ for a given $S_p$ remains an open question. In this paper we suggest $N=S_{p-2}$. The $(S_p,S_{p-2})$ normal MT has several advantages. First of all, (\ref{eq:DerivedSp}) implies $S_{p-2}\Delta \mathbf{v}^{j-1}=1/2\Delta S_{p-1}\mathbf{v}^{j-1}$, while (\ref{eq:Sp}) implies that during the process of evaluating $S_{p}\mathbf{v}^{j-1}$ we have to evaluate $S_{p-1}\mathbf{v}^{j-1}$ anyway, thus generating $\hat{\mathbf{n}}^j$ is basically for free - it neither causes extra computational efforts nor uses extra machine memory. Secondly, $\hat{\mathbf{n}}^j$ have nice geometrical interpretation, namely those are the true discrete normals at the midpoints $S_p\mathbf{v}^{j-1}$ of the polygonal line $S_{p-1}\mathbf{v}^{j-1}(t)$. The latter helps for establishing the following nice fact.

\begin{proposition}\label{prop11}
Let $\mathcal{C}$ be a closed, convex, non-self-intersecting $C^3$ curve, and $p\ge2$. Then, for any choice $\mathbf{v}^0$ of initial data that contains at least $\lceil \frac{p}{2}\rceil+1$ different points, the $(S_p,S_{p-2})$ normal MT is globally well-defined.
\end{proposition}
\begin{proof}
Fix $p\in\NN$, $\mathbf{v}^0\in\mathcal{C}$, respectively $\{s^0 : \mathbf{v}^0=\mathbf{v}(s^0)\}$. The support of the mask of $S_p$ has cardinality $\lceil \frac{p}{2}\rceil+1$, so $\mathbf{v}^0$ should consist of at least that many elements to generate $S_p\mathbf{v}^0$. Fix $i\in\ZZ$ and let $K\in I_{2i}\cup I_{2i+1}$. The $S_p,S_{p-2}$ normal MT belongs to the family of $(S_p,N)$ normal MTs, discussed earlier, for which we already showed that there always exists an intersection point between the normal line through $S_p\mathbf{v}^0_K$ and $\mathcal{C}$. What remains to be checked is that the points in the derived finer sample $\mathbf{v}^1$ are properly ordered (i.e., $s^1$ is monotonically increasing). Due to (\ref{eq:Sp}), $S_p\mathbf{v}^0_K$ is the midpoint of the line segment $(S_{p-1}\mathbf{v}^0)_K(S_{p-1}\mathbf{v}^0)_{K+1}$, while the normal $n^1_K$ is orthogonal to it. Moreover, $\mathcal{C}$, thus also the polyline $S_{p-1}\mathbf{v}^0(t)$ are convex meaning that both the oriented angles $\angle(\Delta S_{p-1}\mathbf{v}^0_{K-1},\Delta S_{p-1}\mathbf{v}^0_{K})$ and $\angle(\Delta S_{p-1}\mathbf{v}^0_{K},\Delta S_{p-1}\mathbf{v}^0_{K+1})$ are positive. Assume that all the detail sequences at all levels for $(S_{p-1},S_{p-3})$ normal MT are nonnegative. Hence, the cone with base $(S_{p-1}\mathbf{v}^0)_K(S_{p-1}\mathbf{v}^0)_{K+1}$, and boundary rays $\{\lambda(\Delta S_{p-2}\mathbf{v}^0_{K})^\perp\;,\;\lambda(\Delta S_{p-2}\mathbf{v}^0_{K+1})^\perp\}$, $\lambda>0$, cuts out the arc $\{\mathbf{v}(s) : s\in[\hat{s}^1_K,\hat{s}^1_{K+1}]\}$ from $\mathcal{C}$, where $\hat{s}^1$ is the arc-length sequence generated from $s^0$ via the $(S_{p-1},S_{p-3})$ normal MT. The ray $\lambda n^1_K$ lies completely within this cone, meaning that there exists a unique intersection point between the ray and the above arc. Now, choosing this point to be $\mathbf{v}^1_K$, we derive $\hat{s}^1_K<s^1_K<\hat{s}^1_{K+1}$, $d^1_K>0$, and the result follows by induction on $p$. The case $p=2$ corresponds to the Chaikin normal MT, considered in \cite{HOS}, and its well-posedness was proven in Theorem 4.1 there. The positivity of the details is trivial.
\end{proof}

For general, not necessarily convex curves, Proposition~\ref{prop11} can be applied locally for disjoint convex/concave pieces and essentially says that if $\mathbf{v}^0$ captures the geometry of $\mathcal{C}$, then the $(S_p,S_{p-2})$ normal MT is well-posed - a very simple and easy-to-check criteria, unlike the one proposed in \cite{HOS}.

\begin{theorem}\label{thm:Smoothness2}
Let $p\ge4$. Let $\mathcal{C}$ be $C^4$, $\mathbf{v}^0\in\mathcal{C}$. If the $(S_p,S_{p-2})$ normal MT for $\mathcal{C}$ with initial data $\mathbf{v}^0$ is well-posed and convergent, then the normal re-parametrization is $C^{3,1-}$.
\end{theorem}
\begin{proof}
It suffices to work with a single refinement step. To simplify notation, let $\mathbf{v}^{j}=\mathbf{v}$, while $\mathbf{v}^{j+1}=\bar{\mathbf{v}}.$ Fix $K\in\ZZ$, let $i\in\ZZ$ be such that $I_K\subset I_{2i}\cup I_{2i+1}$, and denote $h_i:=\|\Delta s\|_{I_i}$. From Theorem~\ref{thm:Smoothness1} and Remark~\ref{remark1} we know that $C'2^{-j}\le h_i\le C2^{-j}$. Let $(\mathcal{T},\mathcal{N})$ be the local frame at $\mathbf{v}(S_ps_K)$ (see Fig.~\ref{fig:LocalFrame}). We estimate $|S_p\big(x(s)\big)_K-x(S_ps_K)|=|S_p\big(x(s)\big)_K|$. Applying $S_p$ to both the sides of (\ref{GridSmth2}), and using that $S_p$ exactly reproduces constants and linear polynomials, we derive

\begin{figure}[htp]
\begin{center}
\scalebox{1} 
{
\begin{pspicture}(0,-2.30)(12.26,1.95)
\psbezier[linewidth=0.04](0.14,-0.09)(0.14,-0.89)(3.4214764,-1.7956831)(4.42,-1.85)(5.418524,-1.9043168)(12.117822,-0.65746766)(11.96,0.33)
\psdots[dotsize=0.12](4.42,-1.85)
\psdots[dotsize=0.12](2.2,-1.35)
\psline[linewidth=0.04cm,arrowsize=0.05291667cm 2.0,arrowlength=1.4,arrowinset=0.4]{->}(0.0,-1.85)(12.24,-1.85)
\psline[linewidth=0.04cm,arrowsize=0.05291667cm 2.0,arrowlength=1.4,arrowinset=0.4]{->}(4.42,-1.85)(4.42,1.93)
\psdots[dotsize=0.12](3.2,1.5)
\psline[linewidth=0.04cm,linestyle=dashed,dash=0.16cm 0.16cm](3.2,1.5)(3.2,-1.85)
\psline[linewidth=0.04cm,linestyle=dashed,dash=0.16cm 0.16cm](3.2,1.5)(4.42,1.5)
\psline[linewidth=0.04cm](3.2,1.5)(2.2,-1.35)
\uput[-90](5.12,-1.85){$\mathbf{v}(S_ps_{K})$}
\uput[-110](2.2,-1.3){$\bar{\mathbf{v}}_K$}
\uput[45](3.2,-1.7){$\tilde{\mathbf{v}}_K$}
\uput[-30](12.24,-1.85){$\mathcal{T}=\mathbf{t}(S_ps_{K})$}
\uput[60](4.42,1.93){$\mathcal{N}=\mathbf{n}(S_ps_{K})$}
\uput[-90](3,-1.85){$S_p\big(x(s)\big)_{K}$}
\uput[0](4.42,1.3){$S_p\big(y(s)\big)_{K}$}
\uput[135](3.2,1.5){$S_p\mathbf{v}_{K}$}
\end{pspicture}
}
\caption{The local frame $(\mathcal{T},\mathcal{N})$ at $\mathbf{v}(S_ps_K)$}\label{fig:LocalFrame}
\end{center}
\end{figure}
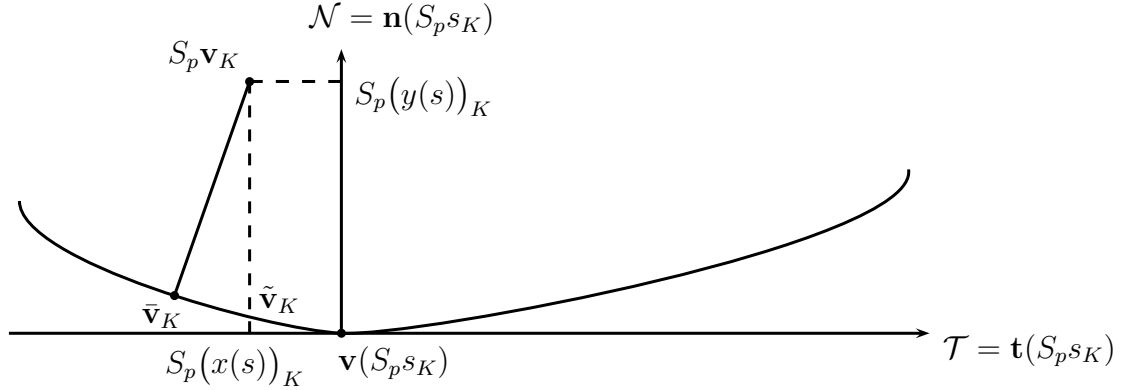
\be\label{GridSmth2a}
S_ps_K=s_i+h_iS_p\left((t-i)|_{\ZZ}\right)_K+(S_p\mathbf{r}_{\cdot,i})_K=s_i+\left(\frac{K}{2}+c_{S_p}-i\right)h_i+\bar{\mathbf{r}}_{K,i},
\ee
where $\|\mathbf{r}_{\cdot,i}\|_{I_i},|\bar{\mathbf{r}}_{K,i}|\le C_22^{-2j}$. Substituting (\ref{GridSmth2a}) into (\ref{GridSmth2}) implies
\be\label{GridSmth2b}
s_k=S_ps_K+\left(k-\frac{K}{2}-c_{S_p}\right)h_i+\mathbf{r}_{k,i}-\bar{\mathbf{r}}_{K,i},\qquad\forall k\in I_i.
\ee
Again, $\|\mathbf{r}_{\cdot,i}-\bar{\mathbf{r}}_{K,i}\|_{I_i}=\mathrm{O}(2^{-2j})$. Combining the latter with (\ref{Taylor expansion for s}), denoting by $\kappa=\mathrm{k}(S_ps_K)$, and applying $S_p$ to both the sides of the formula for $x(s)$ there, we derive
\be\label{estimationSpX}
\begin{split}
S_p(x(s))_K&=(S_p(s-S_ps_K))_K-\frac{\kappa^2}{6}(S_p(s-S_ps_K)^3)_K+S_p(\mathbf{r}_1(s,S_ps_K))_K\\
&=-\frac{\kappa^2}{6}h^3_i(S_p(t-K/2-c_{S_p})^3|_{\ZZ})_K+\mathrm{O}(2^{-4j})=\mathrm{O}(2^{-4j}).
\end{split}
\ee
For the last computation we used (\ref{GridSmth2b}) and $h_i\asymp2^{-j}$ for both $|S_p(\mathbf{r}_1(s,S_ps_K))_K|\le C_{\mathcal{C}}\|s-S_ps_K\|^4_{I_i}=\mathrm{O}(2^{-4j})$ and $(s_k-S_ps_K)^3=(k-K/2-c_{S_p})^3h_i^3+\mathrm{O}(2^{-4j})$. We also used (\ref{eq:P=2 and 3 generalized}). Analogous computations for the $y$-coordinate give rise to
\be\label{estimationSpY}
\begin{split}
S_p(y(s))_K&=\frac{\kappa^2}{6}(S_p(s-S_ps_K)^2)_K+\mathrm{O}(2^{-3j})=\frac{\kappa^2}{6}h_i^2(S_p(t-K/2-c_{S_p})^2|_{\ZZ})_K+\mathrm{O}(2^{-3j})\\
&=\frac{\kappa^2}{6}h_i^2\left((K/2-K/2-c_{S_p}+c_{S_p})^2+\frac{p+1}{16}\right)+\mathrm{O}(2^{-3j})=\frac{\kappa^2(p+1)}{96}h_i^2+\mathrm{O}(2^{-3j}).
\end{split}
\ee

For the $(S_p,S_{p-2})$ normal MT $\bar{\mathbf{v}}_K$ is the intersection point of the bisector of $S_{p-1}\mathbf{v}_KS_{p-1}\mathbf{v}_{K+1}$ and $\mathcal{C}$. The latter can be expressed by the following equation
\be\label{eq:vK}
(\bar{x}_K-S_px_K)S_{p-2}\Delta x_K+(\bar{y}_K-S_py_K)S_{p-2}\Delta y_K=0.
\ee
We will show that it implies
\be\label{sPerturbation}
\omega_K=\mathrm{O}(2^{-4j}).
\ee
From Theorem~\ref{thm:Smoothness1}, we already know that $\omega_K=\mathrm{O}(2^{-3j})$. Applying (\ref{GridSmth2b}) together with (\ref{Taylor expansion for s}), using the exact polynomial reproduction property of $S_p$, (\ref{eq:P=2 and 3 generalized}), and (\ref{eq:c(Sp)}) we deduce
\bea
\bar{x}_K&=&x(\bar{s}_K)=\bar{s}_K-S_ps_K+\mathrm{O}(|\bar{s}_K-S_ps_K|^3)=\omega_K+\mathrm{O}(2^{-9j});\\
\bar{y}_K&=&y(\bar{s}_K)=\mathrm{O}(|\bar{s}_K-S_ps_K|^2)=\mathrm{O}(2^{-6j});\\
\Delta x_k&=&\Delta s_k+\mathrm{O}(2^{-3j})=h_i+\mathrm{O}(2^{-2j})\quad\Rightarrow\quad S_{p-2}\Delta x_K=h_i+\mathrm{O}(2^{-2j});\\
\Delta y_k&=&\frac{\kappa}{2}\Delta s_k(s_{k+1}+s_k-2S_ps_K)+\mathrm{O}(2^{-3j})=\frac{\kappa}{2}h_i^2(2k+1-K-2c_{S_p})+\mathrm{O}(2^{-3j})\\
&\Rightarrow& S_{p-2}\Delta y_K=\frac{\kappa}{2}h_i^2\big(2(K/2+c_{S_{p-2}})+1-K-2c_{S_p}\big)+\mathrm{O}(2^{-3j})=\mathrm{O}(2^{-3j}).
\eea
Note that the choice of the $(S_p,S_{p-2})$ normal MT was crucial for the last estimation and for any other $q\neq p-2$ we have that the coefficient in front of $h_i^2$ in the corresponding expansion of $S_{q}\Delta y_K$ is not zero!

Plugging all the above estimations plus (\ref{estimationSpX}) and (\ref{estimationSpY}) in (\ref{eq:vK}) concludes
$$(\omega_K+\mathrm{O}(2^{-4j}))(h_i+\mathrm{O}(2^{-2j}))+\left(-\frac{\kappa^2(p+1)}{96}h_i^2+\mathrm{O}(2^{-3j})\right)\mathrm{O}(2^{-3j})=0\;
\Rightarrow\;\omega_K=\mathrm{O}(2^{-4j}).$$
What remains is to apply Lemma~\ref{lemma:DRS}.
\end{proof}

Therefore, the use of the operator $S_{p-2}$ for generating the normals in the $S_p$ normal MT further improves the guaranteed smoothness of the normal re-parametrization. From the proof above and the remarks after Lemma~\ref{lemma:Sp} it is clear that $\omega^j_K=c_4\|\Delta s^{j-1}\|^4_{I_i}+\dots$ with $c_4\neq 0$, thus Lemma~\ref{lemma:DRS} is no longer applicable with $k'\ge4$. On the other hand, $S_p$ is approximating and (\ref{eq:DRS}) is a sufficient but not a necessary condition for $C^{k',\alpha'}$ regularity of the data. In \cite{HarizanovPHD} the latter has been replaced by $\|\Delta^n\omega^j\|\le C2^{-(n+\gamma)j}$, which in view of (\ref{eq:DerivedSp}) is nothing but applying Lemma~\ref{lemma:DRS} to the $n$-th derived scheme of $S_p$ and then integrating back. However, such kind of ``precise'' analysis is much more technical and hard to perform, so whether it improves the result above or not is still unknown.

We conclude the section with another result that further illustrates the benefits of choosing $N=S_{p-2}$ in (\ref{Normal2}). Based on it, we build the ``combined prediction'' framework in the next section for improving the detail decay rates of the $S_p$ normal MTs.
\begin{proposition}\label{prop2}
Let $p\ge3$. Let $\mathcal{C}$ be $C^{4}$ strictly convex (i.e., $\mathrm{k}(s)>0$, $\forall s\in[0,L]$), and $\mathbf{v}^0\in\mathcal{C}$ contains at least $\lceil \frac{p}{2}\rceil+1$ different points. Let $\xi^j$ be the arc-length sequence as in (\ref{Normal}), associated to the approximate normals (\ref{Normal2}) with $N=S_{p-2}$. If the $(S_p,S_{p-2})$ normal MT of $\mathcal{C}$, with initial data $\mathbf{v}^0$ is convergent, there is a positive constant $C$ such that
\be\label{eq:NormalProximity}
|\xi^j_K-s^j_K|\le C\|\Delta s^{j-1}\|^2,\qquad\forall j\ge1,\;\forall K\in\ZZ.
\ee
\end{proposition}

\begin{proof}
Proposition~\ref{prop11} implies that the $(S_p,S_{p-2})$ normal MT is well-posed, so Theorem~\ref{thm:Smoothness2} is applicable. Analyzing the proof of the latter, we see that only cubical polynomial reproduction was used, thus (\ref{sPerturbation}) remains valid also for $p=3$. We follow the same notation as there. Let $\tilde{\mathbf{v}}_K$ be generated by $\tilde{\mathbf{n}}_k=\mathbf{n}(S_ps_K)$, as illustrated on Fig.~\ref{fig:LocalFrame}. From (\ref{estimationSpX}), $|\bar{\mathbf{v}}_K-\mathbf{v}(S_ps_K)|\le|\omega_K|$, (\ref{sPerturbation}), $\tilde{s}_K=S_p(x(s))_K$, and using triangle inequality we obtain
\bea
\big||d_K\mathbf{n}(\xi_K)|-|\tilde{d}_K\mathbf{n}(S_ps_K)|\big|&=&\big||S_p\mathbf{v}_K-\bar{\mathbf{v}}_K|-|S_p\mathbf{v}_K-\tilde{\mathbf{v}}_K|\big|\\
&\le&|S_p(x(s))_K|+|\bar{\mathbf{v}}_K-\mathbf{v}(S_ps_K)|+|y(\tilde{s}_K)|\le\mathrm{O}(2^{-4j}).
\eea
To summarize, the last estimation together with (\ref{estimationSpY}) states $$|d_K|=\frac{\kappa^2(p+1)}{96}h_i^2+\mathrm{O}(2^{-3j}),\qquad \big||d_K|-|\tilde{d}_K|\big|=\mathrm{O}(2^{-4j}).$$
For a strictly convex, closed curve $\mathcal{C}$ we have that the curvature $\mathrm{k}(s)$ is uniformly bounded away from zero, meaning that $\mathrm{k}(s)\ge c_{\mathrm{k}}>0$, $s\in[0,L]$, with constant $c_{\mathrm{k}}$ dependent only on $\mathcal{C}$. Thus, $c_2=\kappa^2(p+1)/96\ge c^2_{\mathrm{k}}(p+1)/96>0$. Using the notation from Fig.~\ref{fig:SecondOrderPerturbation}, and applying two cosine theorems for the small and the big triangle, respectively, we conclude

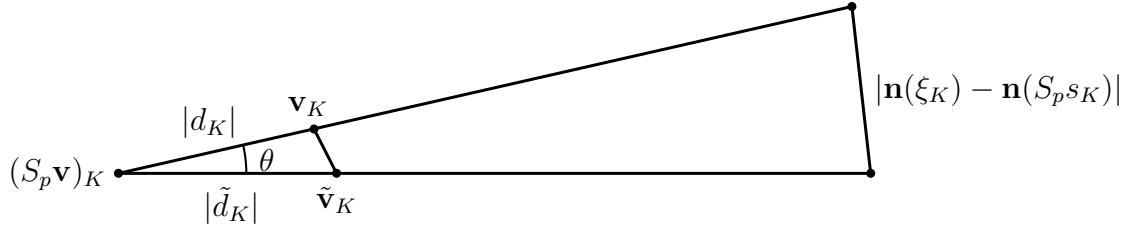
\begin{figure}[htp]
\begin{center}
\scalebox{1} 
{
\begin{pspicture}(0,-1.5)(10.3,1.19)
\psdots[dotsize=0.12](0,-1.11)
\psdots[dotsize=0.12](2.9,-1.11)
\psdots[dotsize=0.12](10,-1.11)
\psline[linewidth=0.04cm](0,-1.11)(10,-1.11)
\psline[linewidth=0.04cm](0,-1.11)(9.7505,1.11)
\psdots[dotsize=0.12](2.6,-0.52)
\psdots[dotsize=0.12](9.7505,1.11)
\psline[linewidth=0.04cm](9.7505,1.11)(10,-1.11)
\psline[linewidth=0.04cm](2.6,-0.52)(2.9,-1.11)
\rput(0,-1.11){
\psarc(0,0){1.7}{0}{13}
\uput[30](1.7,0){$\theta$}
}
\uput[180](0,-1.11){$(S_p\mathbf{v})_{K}$}
\uput[-90](2.9,-1.11){$\tilde{\mathbf{v}}_K$}
\uput[100](2.6,-0.52){$\mathbf{v}_K$}
\uput[-90](1.5,-1.11){$|\tilde{d}_K|$}
\uput[100](1.3,-0.82){$|d_K|$}
\uput[0](9.875,0){$|\mathbf{n}(\xi_K)-\mathbf{n}(S_ps_K)|$}
\end{pspicture}
}
\caption{The estimation for $|\mathbf{n}(\xi_K)-\mathbf{n}(S_ps_K)|$}\label{fig:SecondOrderPerturbation}
\end{center}
\end{figure}
\bea
\cos(\theta)=\frac{2|d_K||\tilde{d}_K|+(|d_K|-|\tilde{d}_K|)^2-|\mathbf{v}_K-\tilde{\mathbf{v}}_K|^2}{2|d_K||\tilde{d}_K|}&=&1+\frac{\mathrm{O}(2^{-8j})}{2c_2^2h_i^4+\mathrm{O}(2^{-5j})}=1+\mathrm{O}(2^{-4j});\\
|\mathbf{n}(\xi_K)-\mathbf{n}(S_ps_K)|^2=2\big(1-\cos(\theta)\big)=\mathrm{O}(2^{-4j})\quad&\Rightarrow&\quad |\mathbf{n}(\xi_K)-\mathbf{n}(S_ps_K)|\le C_{\mathrm{k}}\|\Delta s\|^2_{I_i}.
\eea
Finally, the inequality $\mathrm{k}(s)\ge c_{\mathrm{k}}$, $s\in[0,L]$ and the non-self-intersection assumption on $\mathcal{C}$ also imply that the differentiable function $\mathbf{v}'(s)$ is invertible, and its inverse is Lipschtiz continuous, i.e.,
$$|\mathbf{v}'(s)-\mathbf{v}'(s')|\ge c_{\mathbf{v}'}|s-s'|,\qquad s,s'\in[0,L),$$
where the constant $c_{\mathbf{v}'}>0$ depends solely on $\mathcal{C}$. The argument follows \cite[(6)]{Runborg} and is left to the reader. Therefore
\bea
|\xi_K-S_ps_K|\le\frac{1}{c_{\mathbf{v}'}}|\mathbf{v}'(\xi_K)-\mathbf{v}'(S_ps_K)|=\frac{1}{c_{\mathbf{v}'}}|\mathbf{n}(\xi_K)-\mathbf{n}(S_ps_K)|\le C\|\Delta s\|^2_{I_i}.
\eea
The proof is completed, once again due to (\ref{sPerturbation}).
\end{proof}

An interesting corollary of Proposition~\ref{prop2} is that under its assumptions the choice of generalized normals, corresponding to $\hat{\mathbf{n}}^j_K=\mathbf{n}(\tilde{\xi}^j_K)$, where $(\mathbf{v}(S_ps^{j-1}_K)-S_p\mathbf{v}^{j-1}_K)\parallel\mathbf{n}(\tilde{\xi}^j_K)$, satisfies (\ref{Normal}). Indeed, $\mathcal{C}$ being closed and convex guarantees the existence of such $\tilde{\xi}^j$ for all $j\ge1$, a small modification of the proof of Proposition~\ref{prop2} implies $|\tilde{\xi}^j_K-S_ps^{j-1}_K|\le C\|\Delta s^{j-1}\|^2$, which finally can be combined with (\ref{GridSmth2b}). With such approximate normals, $s^j=S_ps^{j-1}$, thus the transform is always well-posed and convergent. Moreover, $s(t)\in C^{p-1,1}$ and the restrictive role of $P_e$ in the smoothness analysis is completely overtaken. Unfortunately, $\mathbf{n}(\tilde{\xi}^j)$ cannot be generated by the coarse-scale data $\mathbf{v}^{j-1}$, the reconstruction of $\mathcal{C}$ from $\{\mathbf{v}^0,d^1,\dots\}$ is impossible, so such procedure is not a multi-scale transform. On the other hand it gives a criteria how to measure the regularization properties of $(S_p,N)$ normal MT, namely we can check how well $\mathbf{n}(\tilde{\xi}^j)$ are approximated by the vectors $\hat{\mathbf{n}}^j$ generated via $N$.

From this point of view, Proposition~\ref{prop2} establishes that the normals, generated via $S_{p-2}$ are second order perturbations of $\mathbf{n}(\tilde{\xi}^j)$, while using any other $S_q$, $q\neq p-2$ leads to only first order perturbations. Therefore, in some sense $N=S_{p-2}$ is optimal at least in the class $\{S_p\}$. Whether there is an $N$ outside of the B-spline class that leads to even higher-order approximation of $\mathbf{n}(\tilde{\xi}^j)$ is, to the best knowledge of the author, an open question.

\section{Improved detail decay rate. Combined normal MTs}\label{sec4}

\begin{figure}[htp]
\begin{center}
\includegraphics[width=0.7\textwidth]{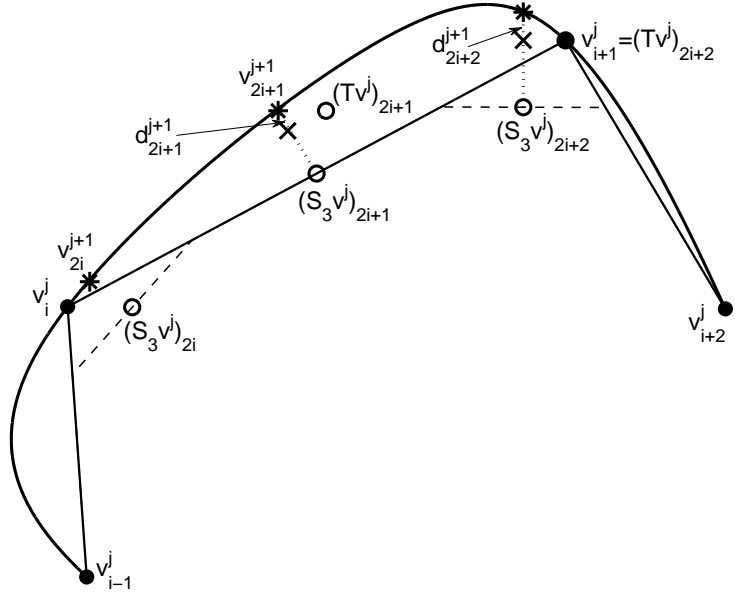}
\caption{ One step of the $(S_3,S_1,T)$ normal MT.} \label{fig:CombinedNormal}
\end{center}
\end{figure}

In this section, influenced by \cite{Oswald} and the local decoupling of parameter and geometry information, we combine two different rules $S$ and $\tilde{S}$ that are used for computing the tangential and the normal component of the predicted point, respectively. We analyze the following procedure: given $\mathbf{v}^{j-1}$ and $K\in\ZZ$, compute $(S\mathbf{v}^{j-1})_K$ and $(\tilde{S}\mathbf{v}^{j-1})_K$. Compute $\hat{\mathbf{n}}^j_K$ via (\ref{Normal2}), where the scheme $N$ is a priori fixed. Then, the predicted point $\hat{\mathbf{v}}_K^{j}$ is the orthogonal projection of $(\tilde{S}\mathbf{v}^{j-1})_K$ onto the normal line $\ell^j_K(t)=(S\mathbf{v}^{j-1})_K+t\hat{\mathbf{n}}^j_K$. In the local frame at $\xi^j_K$, it reads
\be\label{eq:defSTS}
\mathbf{v}_K^j=\hat{\mathbf{v}}_K^{j}+d_K^j\hat{\mathbf{n}}_K^{j};\qquad \hat{x}^j_K=(Sx^{j-1})_K,\quad \hat{y}^j_K=(\tilde{S}y^{j-1})_K.
\ee
We refer to (\ref{eq:defSTS}) as the $(S,N,\tilde{S})$ {\em normal MT} (see Fig.~\ref{fig:CombinedNormal}). In comparison to the $(S,N)$ normal MT, (\ref{eq:defSTS}) involves only additional pre-processing of the data in the normal direction $\hat{\mathbf{n}}^j_K$, subject to $\tilde{S}$, thus the refined data $\mathbf{v}^j$ is the same. Hence, the analysis of well-posedness, convergence, and regularity of the normal re-parameterization depends on $S$ but not on $\tilde{S}$.

\begin{theorem}\label{thm:DetailDecay}
Let $p\ge3$. Let $\mathcal{C}$ be $C^{4}$ strictly convex, and $\mathbf{v}^0\in\mathcal{C}$ contains at least $\lceil \frac{p}{2}\rceil+1$ different points. Let $T$ be a linear subdivision scheme with exact order of polynomial reproduction $P_e\ge4$ and shift parameter $c_{T}=c_{S_p}$. If the $(S_p,S_{p-2},T)$ normal MT of $\mathcal{C}$, with initial data $\mathbf{v}^0$ is convergent, it possesses detail decay rates
$$\|d^j\|=\mathrm{O}(2^{-4j})\quad j\to\infty.$$
\end{theorem}
\begin{proof}
Due to Theorem~\ref{thm:Smoothness2}, we have $\|\omega^j\|\le C2^{-4j}$. Let $n\in\{1,2,3\}$, and denote by $\rho:=2^{n-4}<1$. From (\ref{eq:DerivedSp}), combined with $\|S_q\|=1$, $\forall q\ge0$, it follows that
\be\label{eq:delta s}
\begin{split}
\|\Delta^n s^j\|&\le2^{-n}\|\Delta^n s^{j-1}\|+2^n\|\omega^j\|\le\dots\le2^{-nj}\|\Delta^n s^0\|+C2^n\sum_{\ell=0}^{j-1}2^{-n\ell-4(j-\ell)}\\
&\le 2^{-nj}\left(\|\Delta^n s^0\|+\frac{C2^n}{1-\rho}\right)=C2^{-nj}.
\end{split}
\ee
Again, the constant $C$ changed its value throughout the computation, and this will happen quite often in the proof. What matters is that $C$ remains uniformly bounded and independent on $j$.

It suffices to work with a single refinement step. To simplify notation, let $\mathbf{v}^{j}=\mathbf{v}$, while $\mathbf{v}^{j+1}=\bar{\mathbf{v}}.$   Fix $K\in\ZZ$, let $i\in\ZZ$ be such that $I_K\subset I_{2i}\cup I_{2i+1}$. Let $\xi_K$ be as in (\ref{Normal}), and consider the local frame at it. Due to (\ref{Taylor expansion for s}) and (\ref{eq:delta s}) we have
\be\label{eq:delta x}
\|\Delta^n x\|_{I_i}\le C2^{-nj}.
\ee
The result we want to show is of asymptotical nature while $\|\Delta s^j\|\asymp 2^{-j}$, so without loss of generality we can assume that $x(s)$ is monotonic, thus locally invertible within the interval $s_{I_i}$ and $y(x)=y(s(x))$ makes sense. Applying Lemma~\ref{lem41} for the $C^4$ function $y(x)$, the scheme $T$, and $M=3$ we derive
\be\label{eq:sec62}|y((Tx)_K)-(Ty)_K|\le C \left(\sum_{\nu\in E_3} \prod_{m=1}^3\|\Delta^m x\|_{I_i}^{\nu_m} + \|\Delta x\|_{I_i}^4\right)\le C2^{-4j},
\ee
where $C$ depends only on $\mathcal{C}$ and $T$.
Now, (\ref{GridSmth2}) implies
$$x_k=x_i+(k-i)\|\Delta x\|_{I_i}+ \mathbf{r}_{k,i}, \qquad |\mathbf{r}_{k,i}|\le C_32^{-2j},\qquad\forall k\in I_i,$$
which, due to Lemma~\ref{lemma:Oswald} gives rise to
\be\label{eq:sec63}|(S_px-Tx)_K|\le C2^{-2j}.\ee
From Proposition~\ref{prop2}
\bea
y((S_px(s))_K)&=&y(\bar{s}_K)=\frac{\mathrm{k}(\xi_K)}{2}(\bar{s}_K-\xi_K)^2+\dots\le C2^{-4j};\\
y((Tx(s))_K)&=&y(\bar{s}_K+\mathrm{O}(\|\Delta s\|_{I_i}^{2}))\le C2^{-4j}.
\eea
Finally, combining the last result with (\ref{eq:sec62}) we conclude
\bea
|d_K|=|y((S_px)_K)-(Ty)_k|\le|y((Tx)_K)-(Ty)_K|+|y((S_px)_K)-y((Tx)_K)|\le C2^{-4j}.
\eea
The proof is completed.
\end{proof}

\begin{table}[htb]
\begin{center}
\begin{tabular}{|c|c|c|c|c|c|c|c|c|c|c|}
\hline
\multicolumn{11}{|c|}{Estimation of the detail decay order via computing $-\log_2(\|d^j\|)/j$ for each level $j$ }\\
\hline
Normal MT & $j=1$ & $j=2$ & $j=3$ & $j=4$ & $j=5$ & $j=6$ & $j=7$ & $j=8$ & $j=9$ & $j=10$\\
\hline
\multirow{1}{*}{$(S_3, S_1, T_3)$}  & 5.4040 & 4.7841 & 4.5647 & 4.4079 & 4.3217 & 4.2681 & 4.2313 & 4.2041 & 4.1831 & 4.1649 \\
& 2.0019  & 2.4977 & 2.8437 & 3.0704 & 3.2251 & 3.3354 & 3.4184 & 3.4830 & 3.5347 & 3.5770 \\
\hline
\multirow{1}{*}{$(S_5, S_3, T_5)$}  & 5.2305 & 5.1804 & 4.9641 & 4.7860 & 4.6559 & 4.5629 & 4.4926 & 4.4379 & 4.3938 & 4.3502 \\
& 1.9547  & 3.0385 & 3.3760 & 3.5156 & 3.5967 & 3.6549 & 3.6982 & 3.7318 & 3.7586 & 3.7805\\
\hline
\multirow{1}{*}{$(S_7, S_5, T_7)$}  & 5.0745 & 5.1475 & 4.9214 & 4.7465 & 4.6244 & 4.5357 & 4.4689 & 4.4168 & 4.3747 & 4.3304 \\
& 1.8226  & 3.0567 & 3.3541 & 3.4896 & 3.5766 & 3.6375 & 3.6829 & 3.7182 & 3.7464 & 3.7694 \\
\hline
\end{tabular}
\caption{Numerical verification that the detail decay rate of a combined scheme cannot exceed 4. The unit circle and two sets of initial data - uniform samples of $\pi(x+x^2)$ with step size 0.01 (upper results) and 0.1 (lower results) have been considered.}\label{table53}
\end{center}
\end{table}

In applications, it is highly unlikely that the initial curve $\mathcal{C}$ is convex. However, following the discussion after Proposition~\ref{prop11}, if $\mathbf{v}^0(t)$ approximates $\mathcal{C}$ well, then Theorem~\ref{thm:DetailDecay} implies that away from flat regions of $\mathcal{C}$, where the details are negligible anyway, the combined action of $S_p$ and a suitable $T$ improves twice the compression rate of the pure $S_p$ transform for any closed, non-self-intersecting, $C^4$ curve. For example, the $(S_p,S_{p-2},T)$ normal MT can be successfully incorporated in the adaptive framework \cite{Harizanov}.

Regarding the choice of $T$, suitable candidates seem to be the family $\{T_p\}_{p\ge3}$ of the $2p$-point interpolatory Deslauriers-Dubuc schemes $T_{2p-1}$ \cite{DD}, and the odd counterpart $T_{2p}$ of the dual $2p$-point schemes of Dyn et al. \cite{DynFloaterHormann}. For them we have $c_{T_{2p}}=0$, $c_{T_{2p+1}}=1/4$. Note that, to apply Theorem~\ref{thm:DetailDecay} we need the centered versions $S'_p$ of $S_p$ mentioned in Section~\ref{sec2}, since $c_{S'_{2p}}=0$, $c_{S'_{2p+1}}=1/4$. Until the end of this section we always deal with $S'_p$, but for simplicity we will keep the notation $S_p$.

We already mentioned that the order 2 in (\ref{eq:NormalProximity}) can not be improved, thus neither the order 4 in Theorem~\ref{thm:DetailDecay}. This is experimentally confirmed by the results in Table~\ref{table53}. On the unit circle (which is $C^{\infty}$ and of constant curvature, thus the ``nicest'' possible example) we consider two different irregular initial data samples $\mathbf{v}^0$ and $\hat{\mathbf{v}}^0$, such that the corresponding $s^0$ and $\hat{s}^0$ are just the projection of the quadratic polynomial $\pi(x+x^2)$ on the regular grids $10^{-2}\ZZ$ and $10^{-1}\ZZ$, respectively. Due to periodicity, $\mathbf{v}^0_{i+100}=\mathbf{v}^0_{i}$, and $\hat{\mathbf{v}}^0_{i+10}=\hat{\mathbf{v}}^0_{i}$, so we have 100 distinct points in $\mathbf{v}^0$ and 10 distinct points in $\hat{\mathbf{v}}^0$. For both the initial samples we perform 10 steps of the $(S_p,S_{p-2},T_p)$ normal MTs with $p=3,5,7$.
Note that $P_e(T_p)=p+1\ge4$, so it also does not play a restrictive role for the detail decay rate. At each refinement step we store the quantity $-\log_2(\|d^j\|)/j$ that corresponds to the detail decay order $\mu$. For each of the three transforms we observe the same phenomena. Applied on the denser sample $\mathbf{v}^0$, the values of the logarithm monotonically decrease, while applied on the coarser sample $\hat{\mathbf{v}}^0$, the values of the logarithm monotonically increase with each refinement. Both the sequences tend to 4.

A remarkable corollary of Theorem~\ref{thm:Smoothness1} and Theorem~\ref{thm:DetailDecay} is the case $p=3$, where $T_3$ is the famous 4-point scheme
\bea
(T_3x)_{2i}=x_i,\qquad(T_3x)_{2i+1}=\frac{-x_{i-1}+9x_{i}+9x_{i+1}-x_{i+2}}{16},\qquad i\in\ZZ.
\eea
Note that the $(S_3,S_1,T_3)$ normal MT can be viewed as a particular generalization of $T_3$, a direction that has been actively explored for many years now. This generalized 4-point normal MT performs in an optimal way with respect to both smoothness of the normal re-parameterization and high detail decay rates. Indeed, $\|\omega^j\|\le C2^{-4j}$ was established in the proof of Theorem~\ref{thm:Smoothness2} and the generalized version of Lemma~\ref{lemma:DRS} in \cite{HarizanovPHD} implies $s(t)\in C^{2,1}$.) Furthermore, $P_e(S_3)=2$ so the $(S_3,S_1)$ normal MT has detail decay rate 2, i.e., $\|d^j\|=\mathrm{O}(2^{-2j})$, while $T_3$ generates $C^{1,1}$ limits so the $T_3$ normal MT has detail decay rate 3 up to a logarithmic factor, i.e., $\|d^j\|=\mathrm{O}(j2^{-3j})$ \cite[Section 7.1.2]{DRS}. Thus, their combined action improves the detail decay order of each one of them, if applied alone. The only drawback of the combined transform is that it remains approximating, and thus, twice as many details as for the $T_3$ normal MT are stored. Last, but not least, there is strong numerical evidence that the above phenomena holds in 3D, as well. Indeed, in \cite[Section 4]{Oswald} the $(S_3,S_1,T_3)$ normal MT analogue for triangulated surfaces has been proposed, where $S_3$ has been replaced by the Loop subdivision scheme, while $T_3$ has been replaced by the Butterfly subdivision scheme. Experimented on irregular data samples of the unit sphere, the combined Loop/Butterfly normal MT gives rise to detail decay rate of order 4 away from extraordinary vertices.


\bibliography{BSplineNormalMT}
\bibliographystyle{plain}
\end{document}